\documentclass[12pt]{amsart}
\usepackage{amsfonts}
\usepackage{graphicx}
\usepackage{amsmath}
\usepackage{amssymb}
\usepackage{amsthm}
\usepackage{enumerate}
\usepackage[all, cmtip]{xy} 
\newtheorem{theorem}{Theorem}

\newtheorem{corollary}[theorem]{Corollary}

\newtheorem{lemma}[theorem]{Lemma}

\newtheorem{proposition}[theorem]{Proposition}
\newtheorem{remark}[theorem]{Remark}

\newcommand{\Ker}{\operatorname{Ker}}
\newcommand{\End}{\operatorname{End}}
\newcommand{\Ima}{\operatorname{Im}}
\begin{document}
\title[Automorphism-invariant modules]{Automorphism-invariant modules satisfy the exchange property}
\author{Pedro A. Guil Asensio}
\thanks{First author was partially supported by the DGI (MTM2010-20940-C02-02) 
and by the Excellence Research Groups Program of the S\'eneca Foundation of the Region of Murcia. Part of
the sources of both institutions come from the FEDER funds of the European Union.}
\address{Departamento de Mathematicas, Universidad de Murcia, Murcia, 30100, Spain}
\email{paguil@um.es}
\author{Ashish K. Srivastava}
\address{Department of Mathematics and Computer Science, St. Louis University, St.
Louis, MO-63103, USA}
\email{asrivas3@slu.edu}

\begin{abstract}
Warfield proved that every injective module has the exchange property. This was generalized by Fuchs who showed that quasi-injective modules satisfy the exchange property. We extend this further and prove that a 
module invariant under automorphisms of its injective hull satisfies the exchange property. We also show that automorphism-invariant modules are clean and that directly-finite automorphism-invariant modules satisfy the internal cancellation and hence the cancellation property.
\end{abstract}

\maketitle

\section{Introduction.}

\noindent A module $M$ which is invariant under automorphisms of its injective hull is called an {\it automorphism-invariant module}. This class of modules was first studied by Dickson and Fuller in \cite{DF} for the particular case of finite-dimensional algebras over fields $\mathbb F$ with more than two elements. Clearly, every quasi-injective and hence injective module is automorphism-invariant. Recently, it has been shown in \cite{ESS} that a module $M$ is automorphism-invariant if and only if every monomorphism from a submodule of $M$ extends to an endomorphism of $M$. 
And it has been shown in \cite{GS} that any automorphism-invariant module $M$ is quasi-injective provided that $\End(M)$ has no homomorphic image isomorphic to the field of two elements $\mathbb{F}_2$, thus extending the results obtained by Dickson and Fuller in \cite{DF}. However, Remark \ref{nonqi} below shows that  an automorphism-invariant module does not need to be quasi-injective if we do not assume this additional hypothesis. For more details on automorphism-invariant modules, see \cite{ESS, GS, LZ, SS1, SS2, Teply}. 

On the other hand, it was shown by Fuchs \cite{Fuchs} that every quasi-injective module satisfies the exchange property, thus extending a previous result of Warfield for injective modules \cite{Warfield1}. Recall that the notion of exchange property for modules was introduced by 
Crawley and J\'{o}nnson \cite{CJ}. A right
 $R$-module $M$ is said to satisfy the {\it exchange property} if for every right $R$-module $A$ and any two direct sum decompositions $A=M^{\prime}\oplus N=\oplus_{i\in \mathcal I}A_{i}$ with $M^{\prime} \simeq M$, there exist submodules $B_i$ of $A_i$ such that $A=M^{\prime} \oplus (\oplus_{i \in \mathcal I}B_i)$.
If this hold only for $|\mathcal I|<\infty$, then $M$ is said to satisfy the finite exchange property. Crawley and J\'{o}nnson raised the question whether the finite exchange property always implies the full exchange property but this question is still open. 

A ring $R$ is called an {\it exchange ring} if the module $R_R$ (or $_RR$) satisfies the (finite) exchange property. Goodearl \cite{Goodearl} and Nicholson \cite{Nichol} provided several equivalent characterizations for a ring to be an exchange ring. Warfield \cite{Warfield2} showed that exchange rings are left-right symmetric and that a module $M$ has the finite exchange property if and only if $\End(M)$ is an exchange ring.     

The goal of this paper is to show that, besides Remark \ref{nonqi} shows that an automorphism-invariant module $M$ does not need to be quasi-injective in general, it shares several important decomposition properties with quasi-injective modules. Namely, the endomorphism ring of an automorphism-invariant module $M$ is always a von Neumann regular ring modulo its Jacobson radical $J$, idempotents lift modulo $J$ and $J$ consists of those endomorphism of $M$ which have essential kernel (see Proposition \ref{semiregular}). As a consequence, we show in Theorem \ref{main} that any automorphism-invariant module satisfies the exchange property. 

A module $M$ is said to have the {\it cancellation property} if whenever $M\oplus A\cong M\oplus B$, then $A\cong B$. A module $M$ is said to have the {\it internal cancellation property} if whenever $M=A_1\oplus B_1\cong A_2\oplus B_2$ with $A_1\cong A_2$, then $B_1\cong B_2$. A module with the cancellation property always satisfies the internal cancellation property but the converse need not be true, in general. Fuchs \cite{Fuchs} had shown that if $M$ is a module with the finite exchange property, then $M$ has the cancellation property if and only if $M$ has the internal cancellation property. 

A module $A$ is said to have the {\it substitution property} if for every module $M$ with decompositions $M=A_1\oplus H=A_2\oplus K$ with $A_1\cong A\cong A_2$, there exists a submodule $C$ of $M$ (necessarily $\cong A$) such that $M=C\oplus H=C\oplus K$.

A module $M$ is called {\it directly-finite} (or {\it Dedekind-finite}) if $M$ is not isomorphic to a proper summand of itself. A ring $R$ is called directly-finite if $xy=1$ implies $yx=1$ for any $x, y\in R$. It is well-known that a module $M$ is directly-finite if and only if its endomorphism ring $\End(M)$ is directly-finite. 

In general, we have the following hierarchy;\\
\\
substitution$\implies$ cancellation$\implies$internal cancellation$\implies$directly-finite

In this paper, we show in Corollary \ref{internal} that for an automorphism-invariant module the above four notions are equivalent. 

\bigskip


Throughout this paper, $R$ will always denote an associative ring with identity element and modules will be right unital. $J(R)$ will denote the Jacobson radical of the ring $R$. We will use the notation $N\subseteq_e M$ to stress that $N$ is an essential submodule of $M$. We refer to \cite{AF} and \cite{MM} for any undefined notion arising in the text.

\section*{Results.}

\noindent It was proved by Faith and Utumi \cite{FU} that if $M$ is a quasi-injective module and $R=\End(M)$, then $J(R)$ consists of all endomorphisms of $M$ having essential kernel and $R/J(R)$ is a von Neumann regular ring. Later, Osofsky \cite{O} proved that $R/J(R)$ is right self-injective too.

Let us fix some notation that we will follow along this paper. Let $M$ be a module and $E=E(M)$, its injective hull. Call $R=\End(M)$, $S=\End(E)$ and $J(S)$, the Jacobson radical of $S$. Let us denote by $\varphi: R\rightarrow S/J(S)$ the ring homomorphism which assigns to every $r\in\End(M)$ the element $s+J(S)$, where $s:E\rightarrow E$ is an extension of $r$ (see e.g. \cite{GS}). Set $\Delta=\{r \in R: \Ker(r)\subseteq_e M\}$. As $J(S)$ consists of all endomorphisms $s\in S$ having essential kernel, we get that $\Ker(\varphi)=\Delta$ and thus, $\varphi$ factors through an injective ring homomorphism $\psi: R/\Delta \rightarrow S/J(S)$. 

Let $N$ be a submodule of a module $M$. A submodule $L$ of $M$ is called a complement of $N$ in $M$ if it is maximal with respect to the condition $N\cap L=0$. Using Zorn's Lemma it is very easy to check that any submodule $N$ of $M$ has a (not necessarily unique) complement $L$ and $N\oplus L\subseteq_e M$ (see e.g. \cite[Proposition 5.21]{AF}). 

In our first proposition, we extend the above mentioned result of Faith and Utumi to automorphism-invariant modules.

\begin{proposition}\label{semiregular}
Let $M$ be an automorphism-invariant module. Then $\Delta=J(R)$ is the Jacobson radical of $R$, $R/J(R)$ is a von Neumann regular ring and idempotents lift modulo $J(R)$.
\end{proposition}

\begin{proof}
Let $r\in R$ and let $s\in E$ be an extension of $R$. Call $K=\Ker(r)$. Let $L$ be a complement of $K$ in $M$. Then $K\oplus L\subseteq_{e} M$ and thus, $E=E(K)\oplus E(L)$.  Set $g\in S$ as $g|_{E(K)}=0$ and $g|_{E(L)}=s|_{E(L)}$. Then $(g-s)|_{K\oplus E(L)}=0$ and therefore, $g-s \in J(S)$. Therefore, $1-(g-s)$ is an automorphism of $E$. Since $M$ is automorphism-invariant, $(1-(g-s))(M)\subseteq M$. This means that $(g-s)(M)\subseteq M$. Now as $s$ is an extension of $r\in R$, we have $s(M)\subseteq M$. This yields $g(M)\subseteq M$. 

As $L\cap \Ker(g)=0$, $g|_{E(L)}$ is a monomorphism. Let $E'=\Ima(g)=\Ima(g|_{E(L)})$. Then $E'\cong E(L)$ is injective. Moreover, as $g|_{E(L)}:E(L)\rightarrow E'$ is an isomorphism, there exists a homomorphism $h:E'\rightarrow E(L)$  such that $h\circ g\circ u=u\circ1_{E(L)}$ and thus, $u\circ h\circ g=u\circ \pi$, where $u:E(L)\rightarrow E$ and $\pi:E\rightarrow E(L)$ are the inclusion and projection associated to the decomposition $E=E(K)\oplus E(L)$. 
 As $L$ is essential in $E$, $g(L)$ is essential in $E'$ and thus, $N=M\cap g(L)$ is also essential in $E'$, because $M$ is essential in $E$. 
As $M$ is automorphism-invariant, this means that the monomorphism $h|_{N}:N\rightarrow L\subseteq M$ extends to an endomorphism $t:M\rightarrow M$ (see \cite{ESS}). 
Let $t':E\rightarrow E$ be an extension of $t$. As $N$ is essential in $E'=\Ima(g)$, $g^{-1}(N)$ is essential in $E$ and thus, $N'=(K\oplus L)\cap g^{-1}(N)$ is also essential in $E$.
Let us choose an element $x\in N'$. By definition of $N'$, we get that $g(x)\in N$. Moreover, we can write $x=k+l$ with $k\in K$ and $l\in L$. Therefore, $g(l)=g(k)+g(l)=g(x)\in N\subseteq M$. Thus, $t'\circ g(x)=t'\circ g(l)=t\circ g(l)=h\circ g(l)=l$. 
We have shown that $t'\circ g|_{N'}=u\circ \pi|_{N'}$ and, as $N'$ is essential in $E$, this means that $t'\circ g + J(S)=u\circ \pi + J(S)$.
Therefore, $s\circ t'\circ s + J(S) = g\circ t'\circ g + J(S) = g\circ u\circ \pi +J(S)= g+J(S)=s+J(S)$. But then, we deduce that $\psi((r\circ t\circ r)+\Delta)=(s\circ t'\circ s) + J(S)=s+J(S)=\psi(r+\Delta)$ and, as $\psi$ is injective, we get that $(r\circ t\circ r)+\Delta=r+\Delta$. This shows that $R/\Delta$ is von Neumann regular.

Since $R/\Delta$ is von Neumann regular, $J(R/\Delta)=0$. This gives $J(R)\subseteq \Delta$. Now let $a\in \Delta$. Since $\Ker(a) \cap \Ker(1-a)=0$ and $\Ker(a) \subseteq_{e} M$, $\Ker(1-a)=0$. Hence $(1-a)$ is an isomorphism from $M$ to $(1-a)(M)$. Since $M$ is automorphism-invariant, $M$ satisfies the ($C_2$) property (see \cite{ESS}), that is, submodules isomorphic to a direct summand of $M$ are direct summands. Therefore, $(1-a)(M)$ is a direct summand of $M$. However, $(1-a)(M) \subseteq_{e} M$ since $\Ker(a) \subseteq (1-a)(M)$. Thus $(1-a)(M)=M$ and therefore, $1-a$ is a unit in $R$. This means that $a\in J(R)$ and therefore $\Delta\subseteq J(R)$. Hence $J(R)=\Delta$. 

This shows that $R/J(R)$ is a von Neumann regular ring. Now, we proceed to show that idempotents lift modulo $J(R)$. Let $e'+J(R)$ be an idempotent in $R/J(R)$. Let $f'+J(S)=\psi(e'+J(R))$. Then $f'+J(S)$ is an idempotent in $S/J(S)$. Since idempotents lift modulo $J(S)$, there exists an idempotent $f$ in $S$ such that $f'=f+j$ with $j\in J(S)$. Now, $1-j$ is a unit in $S$, and so $M$ is invariant under $1-j$ and hence $j(M)\subseteq M$. And thus, $f(M)\subseteq f'(M)+j(M)\subseteq M$. This means that $e=f|_M$ belongs to $R=\End(M)$ and it is an idempotent since so is $f$. By construction, $\psi(e+J(R))=f+J(S)=f'+J(S)=\psi(e'+J(R))$. And, as $\psi$ is an injective homomorphism, we deduce that $e+J(R)=e'+J(R)$.  This shows that idempotents lift modulo $J(R)$.
\end{proof}

\begin{remark}\label{nonqi}
Note that in the above proposition unlike quasi-injective modules, $R/J(R)$ need to be right self-injective. For example, let $R$ be the ring of all eventually constant sequences $(x_n)_{n \in \mathbb{N}}$ of elements in  $\mathbb{F}_2$. Then $E(R_R)=\prod_{n\in \mathbb{N}}\mathbb{F}_2$, and it has only one automorphism, namely the identity automorphism. Thus, $R_R$ is automorphism-invariant but it is not self-injective. Also, as $R$ is von Neumann regular, $J(R)=0$. This means $R/J(R)$ is not self-injective.  
\end{remark}

As a consequence of the above proposition, we are ready to generalize the results of Warfield and Fuchs. 

\begin{theorem}\label{main}
An automorphism-invariant module satisfies the exchange property.
\end{theorem}

\begin{proof}
Let $M$ be an automorphism-invariant module and $R=\End(M)$. By the above proposition, $R/J(R)$ is a von Neumann regular ring and idempotents lift modulo $J(R)$. Such rings are called semiregular rings (or f-semiperfect rings). Nicholson in \cite[Proposition 1.6]{Nichol} proved that every semiregular ring is an exchange ring. Hence $R$ is an exchange ring. This proves that $M$ has the finite exchange property. 

Now, we know that $M=P\oplus Q$ where $Q$ is quasi-injective and $P$ is a square-free module \cite[Theorem 3]{ESS}. Since a direct summand of a module with the finite exchange property also has the finite exchange property, $P$ has the finite exchange property. But for a square-free module, the finite exchange property implies the full exchange property \cite[Theorem 9] {Nielsen}. So $P$ has the full exchange property. We already know that every quasi-injective module has the full exchange property, so $Q$ has the full exchange property. Since a direct sum of two modules with the full exchange property also has the full exchange property, it follows that $M$ has the full exchange property.         
\end{proof}

We would like to thank Professor Yiqiang Zhou for kindly pointing out the next corollary. Recall that a ring $R$ is called a {\it clean ring} if each element $a\in R$ can be expressed as $a=e+u$ where $e$ is an idempotent in $R$ and $u$ is a unit in $R$. A module $M$ is called a clean module if $\End(M)$ is a clean ring. It was shown in \cite{CKLNZ} that continuous modules are clean. 

\begin{corollary}
Automorphism-invariant modules are clean.
\end{corollary}

\begin{proof}
Let $M$ be an automorphism-invariant module. We have $M=P\oplus Q$ where $Q$ is quasi-injective and $P$ is a square-free module \cite[Theorem 3]{ESS}. By \cite{CKLNZ}, $Q$ is a clean module. By the above theorem, $\End(P)$ is an exchange ring. We know that the idempotents in $\End(P)/J(\End(P))$ are central (see \cite{MM}). Thus $\End(P)/J(\End(P))$ is an exchange ring with all idempotents central. Hence $\End(P)/J(\End(P))$ is a clean ring. Since idempotents lift modulo $J(\End(P))$ by the Proposition \ref{semiregular}, it follows that $\End(P)$ is a clean ring. Thus $P$ is a clean module and hence $M$ is clean.  
\end{proof}

We recall that a module $M$ is called indecomposable if its only direct summands are $0$, and $M$. And a decomposition of a module $M$ as a direct sum of indecomposable modules, say $M=\oplus_{i\in I}M_i$, is said to complement (maximal) direct summands provided that for any (resp. maximal) direct summand $N$ of $M$ there exists a subset $I'\subseteq I$ such that $M=N\oplus(\oplus_{i\in I'}M_i)$ (see \cite[\S12]{AF}). In particular, it is shown in \cite[12.4]{AF} that if $M=\oplus_{i\in I}M_i$ is a decomposition that complements (maximal) direct summands, then this decomposition is unique up to isomorphisms in the sense of the Krull-Remak-Schmidt Theorem. On the other hand, it has been shown in \cite[Theorem 2.25]{MM} that a decomposition of a module $M=\oplus_{i\in I}M_i$, satisfying that $\End(M_i)$ is a local ring for every $i\in I$, complements maximal direct summands if and only if $M$ is an exchange module. We can then show.

\begin{corollary}
Let $M$ be an automorphism-invariant module. If $M$ is a direct sum of indecomposable modules, then this decomposition complements direct summands.
\end{corollary}
\begin{proof}
We know that $\End(M)/J(\End(M))$ is von Neumann regular and therefore, so is $\End(M_i)/J(\End(M_i))$, for every $i\in I$. And, as $M_i$ is indecomposable, this means that $\End(M_i)$ is local. The result now follows from the previous comments.
\end{proof}

Next, we state a useful lemma for directly-finite automorphism-invariant modules.

\begin{lemma} \label{df}
Let $M$ be an automorphism-invariant module. If $M$ is directly-finite, then so is $E=E(M)$.
\end{lemma}
\begin{proof}
Let $s,t\in S$ such that $t\circ s=1_E$. This means that $s:E\rightarrow E$ is a monomorphism. Call $N=M\cap (s^{-1}(M))$. Then $N$ is an essential submodule of $M$ and $s|_N:N\rightarrow M$ is a monomorphism. Therefore, it extends to an endomorphism of $r:M\rightarrow M$ by \cite{ESS}. Let $s':E\rightarrow E$ be an extension of $r$. Then $s|_N=s'|_N$ and thus, $s+J(S)=s'+J(S)=\psi(r+J(R))$. Moreover, as $N$ is essential in $M$ and $r|_N=s|_M$, we get that $r$ is also a monomorphism. As $R/J(R)$ is von Neumann regular and idempotents lift modulo $J(R)$, there exists an idempotent $e\in R$ such that $(Re+J(R))/J(R)=(Rr+J(R))/J(R)$. But this means that, in particular, $r=r'e+j$ for some $r'\in R$ and $j\in J(R)$. 
As $J(R)$ consists of $r\in\End(M)$ with essential kernel by Proposition \ref{semiregular}, we deduce that $K=\Ker(j)$ is essential in $M$. Moreover, $r(m)=r'e(m)+j(m)=r'e(m)$ for every $m\in K$. And, as $r$ is injective, this means that $K\cap \Ima(1-e)=0$. But, $K$ being essential in $M$, this implies that $1-e=0$ and thus, $e=1_M$. Therefore, we deduce that $(Rr+J(R))/J(R)=R/J(R)$ and thus, there exists an element $r''\in R$ such that $1-r''\circ r\in J(R)$ and thus, $r''\circ r$ is an automorphism. As we are assuming that $M$ is directly finite, this means that $r$ must be an automorphism and thus, $s+J(S)=\psi(r+J(R))$ is also an automorphism. Therefore, $s$ is an automorphism and, as we are assuming that $t\circ s=1_S$, we get that $s\circ t=1_S$ too. This shows $S$ is directly-finite and consequently, $E$ is directly-finite. 
\end{proof}

Recall that a ring $R$ is called {\em unit-regular} if, for every element $x\in R$, there exists a unit $u\in R$ such that $x=xux$. We can now prove.

\begin{theorem}
Let $M$ be an automorphism-invariant module. If $M$ is directly-finite, then $R/J(R)$ is unit-regular.
\end{theorem}

\begin{proof}
We are going to adapt the proof of Proposition \ref{semiregular}. Let $r\in R$ and let us construct $s,g\in S$ as in Proposition \ref{semiregular}.
We know that, if we call $E'=\Ima(g)$, $E'\cong E(L)$ is an injective submodule of $E$. So there exists a submodule $E''$ of $E$ such that $E=E'\oplus E''$. We have that $E$ is directly finite by the above lemma. Since a directly-finite injective module satisfies the internal cancellation property, we have $E''\cong E(K)$. Let $\varphi:E''\rightarrow E(K)$ be an isomorphism. 
Define $t:E\rightarrow E$ as follows: $t|_{E''}=\varphi$ whereas $t|_{E'}=h$. 
Clearly, $t\circ g=u\circ \pi$, where $u:E(L)\rightarrow E$ and $\pi:E\rightarrow E(L)$ are the inclusion and projection associated to the decomposition $E=E(K)\oplus E(L)$. 
Moreover, $t$ is clearly an automorphism and this implies that $t(M)\subseteq M$. 
Call $t':M\rightarrow M$ the restriction of $t$ to $M$. Then $t'$ is a monomorphism and, as $R/J(R)$ is von Neumann regular and idempotents lift modulo $J(R)$, there exists an idempotent $e\in R$ such that $(Rt'+J(R))/J(R)=(Re+J(R))/J(R)$. 
Once again as in the proof of Lemma \ref{df}, we get that $t'$ must be an automorphism. Finally, 
$$
\begin{array}{c}
\psi((r\circ t'\circ r)+J(R))=(s\circ t'\circ s)+J(S) =\\ (g\circ t'\circ g)+J(S) = g+J(S) = \psi(r+J(R))
\end{array}
$$
and, as $\psi$ is injective, we get that $(r\circ t'\circ r)+J(R)=r+J(R)$. As $t'\in R$ is an automorphism, this shows that $R/J(R)$ is unit-regular
\end{proof}

The example given in Remark \ref{nonqi} shows that if $M$ is a directly-finite automorphism-invariant module, even then $R/J(R)$ need not be self-injective. 

\bigskip

As a consequence of the above theorem, we have the following

\begin{corollary} \label{internal}
Let $M$ an automorphism-invariant module. Then the following are equivalent;
\begin{enumerate}[(i)]
\item $M$ is directly finite.
\item $M$ has the internal cancellation property.
\item $M$ has the cancellation property.
\item $M$ has the substitution property.
\end{enumerate}
\end{corollary}

\begin{proof}
(i)$\implies$(ii). If $M$ is a directly-finite automorphism-invariant module, then by the above theorem $R/J(R)$ is unit-regular and therefore, by \cite{KL}, $R$ is a ring with the internal cancellation. Now it follows from Guralnick and Lanski \cite{GL}, that $M$ has the internal cancellation property. This shows (i) implies (ii). 

(ii)$\implies$(iii). Since we have already shown that an automorphism-invariant module satisfies the exchange property, (ii) implies (iii), by Fuchs \cite{Fuchs}. 

(iii)$\implies$(i). This implication holds for any module. 

Thus we have established the equivalence of (i), (ii) and (iii). For a module with the finite exchange property, the equivalence of (ii) and (iv) follows from \cite{Yu}. This completes the proof.
\end{proof}

\bigskip

\bigskip

\end{document}